\documentclass{article}

\usepackage[american]{babel}
\usepackage{amsmath}
\usepackage{amsthm}
\usepackage{amsfonts}
\usepackage{blkarray}
\usepackage{amssymb}
\usepackage{graphicx}
\usepackage{verbatim}
\usepackage{color}
\usepackage{xcolor}
\usepackage{fancybox}
\usepackage{float}
\usepackage[active]{srcltx}
\usepackage{enumerate}
\usepackage{mathtools}
\usepackage{listings}
\usepackage{color}
\usepackage[unicode]{hyperref}
\usepackage{caption}
\usepackage{subcaption}
\usepackage[numbers]{natbib}
\usepackage[algo2e,ruled]{algorithm2e}
\usepackage{cleveref}

\newtheorem{theorem}{Theorem}
\newtheorem{lemma}{Lemma}

\newtheorem{definition}{Definition}

\newtheorem{proposition}{Proposition}

\newcommand{\R}{\mathbb{R}}
\newcommand{\Rn}{\R^n}
\newcommand{\nn}{{n \times n}}

\newcommand{\Snn}{\mathbb{S}^\nn}
\newcommand{\defn}{\mathrel{\mathop :}=}
\newcommand{\st}{\text{s.t.}}

\def\0{{\bf 0}}
\def\1{{\bf 1}}

\def \bmat{\left[\begin{matrix}}
	\def \emat{\end{matrix}\right]}
\def \bvec{\left(\begin{matrix}}
	\def \evec{\end{matrix}\right)}
\def \QED{\begin{flushright}\Halmos\end{flushright}\end{proof}}
\def \defeq{\mathrel{\mathop{:}}=}

\def \xbar{\bar{x}}

\def \grad{\nabla}

\def \gp3d{\grad p_3(d)}
\def \Hess{\nabla^2}
\def \beq{\begin{equation}}
\def \eeq{\end{equation}}
\def \baeq{\begin{equation*}\begin{aligned}}
\def \eaeq{\end{aligned}\end{equation*}}
\newcommand{\baeql}[1]{\begin{equation}\label{#1}\begin{aligned}}
\def \eaeql{\end{aligned}\end{equation}}

\newcommand{\jeff}[1]{{#1}}

\title{Higher-Order Newton Methods\\ with Polynomial Work per Iteration}
\author{Amir Ali Ahmadi\thanks{Princeton University, Operations Research and Financial Engineering. AAA and AC were partially supported by the MURI award of the AFOSR and the Sloan Fellowship.}, Abraar Chaudhry\footnotemark[1], Jeffrey Zhang\thanks{Yale University, \jeff{Department of} Biomedical Informatics and Data Science.}}
\date{}

\begin{document}

\maketitle

\begin{abstract}

We present generalizations of Newton's method that incorporate derivatives of an arbitrary order $d$ but maintain a polynomial dependence on dimension in their cost per iteration. At each step, our $d$\textsuperscript{th}-order method uses semidefinite programming to construct and minimize a sum of squares-convex approximation to the $d$\textsuperscript{th}-order Taylor expansion of the function we wish to minimize. We prove that our $d$\textsuperscript{th}-order method has local convergence of order $d$. This results in lower oracle complexity compared to the classical Newton method. We show on numerical examples that basins of attraction around local minima can get larger as $d$ increases. Under additional assumptions, we present a modified algorithm, again with polynomial cost per iteration, which is globally convergent and has local convergence of order $d$.




\end{abstract}


\section{Introduction}\label{Sec: Introduction}

Newton's method is perhaps one of the most well-known and prominent algorithms in optimization. In its attempt to minimize a function $f:\Rn\rightarrow\R$, this algorithm 
replaces $f$ with its second-order Taylor expansion at an iterate $x_k\in\Rn$ and defines the next iterate $x_{k+1}$ to be a critical point of this quadratic approximation. This critical point coincides with a minimizer of the quadratic approximation in the case where the Hessian of $f$ at $x_k$ is positive semidefinite. 

The work required in each iteration of Newton's method consists of solving a system of linear equations which arises from setting the gradient of the quadratic approximation to zero. This can be carried out in time that grows polynomially with the dimension $n$. Perhaps the most well-known theorem about the performance of Newton's method is its \emph{local quadratic convergence}. More precisely, under the assumptions that the second derivative of $f$ is locally Lipschitz around a local minimizer $x^*$, and that the Hessian at $x^*$ is positive definite, there exists a full-dimensional basin around $x^*$ and a constant $c$,
such that if $x_0$ is in this basin, one has $$\|x_{k+1} - x^*\| \le c\|x_k - x^*\|^2$$ for all $k\geq 0$.
We note however that Newton's method is in general not globally convergent. Lack of global convergence can occur even when in addition to the previous assumptions, $f$ is assumed to be strongly convex (see, e.g., Example~\ref{Ex: 363 function} in Section~\ref{Sec: Numerical examples}).

As higher-order Taylor expansions provide closer local approximations to the function $f$, it is natural to ask why Newton's method limits the order of Taylor approximation to 2. The main barrier to higher-order Newton methods is the computational burden associated with minimizing polynomials of degree larger than 2 which would arise from higher-order Taylor expansions. For instance, any of the following tasks that one could consider for each iteration of a higher-order Newton method are in general NP-hard:


\begin{enumerate}[(i)]
\item finding a global minimum of polynomials of degree even\footnote{Note that odd-degree polynomials are unbounded below.} and at least 4 (see, e.g.,~\cite{murty1987some}), 
    \item finding a local minimum of polynomials of degree at least 4 (see~\cite[Theorem 2.1]{ahmadi2022QPs}), 
    \item finding a second-order point (i.e., a point where the gradient vanishes and the Hessian is positive semidefinite) of polynomials of degree at least 4 (see~\cite[Theorem 2.2]{ahmadi2022complexity}),
    \item finding a critical point (i.e., a point where the gradient vanishes) of polynomials of degree at least 3 (see~\cite[Theorem 2.1]{ahmadi2022complexity}).  
\end{enumerate}


In addition to matters related to computation, there are geometric distinctions between Newton's method and higher-order analogues of it. For example, even when the function $f$ is strongly convex and the starting iterate is arbitrarily close to its minimizer, Taylor expansions of \jeff{even} degree and larger than 2 may not be bounded below. One can see this by examining the strongly convex univariate function $f(x) = x^2 - x^4 + x^6$ and its 4\textsuperscript{th} order Taylor expansion near the origin.


Despite these barriers, the question of whether one can make higher-order Newton methods tractable and in some way superior to Newton's method has been considered at least since the work of Chebyshev \cite{chebyshev} (see Section~\ref{SSec: Related Work} for more recent literature). More specifically, the question that is of interest to us is whether it is possible to design a higher-order Newton method (i.e., a method which utilizes a Taylor expansion of degree $d>2$ in each iteration) in such a way that (i) the work per iteration grows polynomially with the dimension, and (ii) the local order of convergence grows with $d$, hence requiring fewer function evaluations as $d$ increases. In this paper, we show that this is indeed possible (\Cref{alg:main} and \Cref{thm:rate}).



Our algorithm relies on sum of squares techniques in optimization~\cite{pablothesis},~\cite{Lasserre2000} and semidefinite programming and does not require the function $f$ to be convex.
For any fixed degree $d$, our approach is to approximate the $d$-th order Taylor expansion of $f$ with an ``sos-convex'' polynomial (see Section~\ref{Sec: Preliminaries} for a definition). 
Sos-convex polynomials form a subclass of convex polynomials whose convexity has an explicit algebraic proof.  
One can then use a first-order sum of squares relaxation to minimize this sos-convex polynomial. It turns out that both the task of finding a suitable sos-convex polynomial and that of minimizing it can be carried out by solving two semidefinite programs whose sizes are polynomial in the dimension $n$ (in fact of the same order as the number of terms in the Taylor expansion). As is well known, semidefinite programs can be solved to arbitrary accuracy in polynomial time; see~\cite{vandenberghe1996semidefinite} and references therein.  

We work with sos-convex polynomials instead of general convex polynomials since the latter set lacks a tractable description \cite{ahmadi2013np}, and the former, as we show, turns out to be sufficient for achieving an algorithm with superlinear local convergence. Our sum of squares based algorithm works for higher-order Newton methods of any order $d$ and can be easily implemented using any sum of squares parser (e.g., YALMIP~\cite{lofberg2004yalmip} or SOSTOOLS~\cite{prajna2002introducing}). This is in contrast to previous work where implementable algorithms have been worked out only for $d=3$~; see \cite{nesterov2021implementable},~\cite[Sect. 1.5]{doikovthesis},~\cite[Sect. 5]{grapiglia2022tensor}. While we present our algorithms in the unconstrained case, they can be readily implemented in the presence of sos-convex constraints (such as linear constraints or convex quadratic constraints). We note, however, that our interest in this paper is only on generalizing Newton's method in terms of its convergence order and polynomial work per iteration, and not on the practical aspects of implementation.
Designing more scalable algorithms for semidefinite programs is an active area of research~\cite{majumdar2020recent,yurtsever2021scalable}. In addition, we believe that there are promising future research directions which could make our algorithms more practical at larger scale (see \Cref{sec:future}).

\subsection{Related Work}\label{SSec: Related Work}

Over the years, there have been many adaptations of and extensions to Newton's method. A primary example is the pioneering work of Nesterov and Polyak \cite{nesterov2006cubic}, where the idea of Newton's method with cubic regularization was introduced. We do not review the large literature that emerged from this work since the order $d$ of Taylor expansion in this line of work is still equal to 2, and hence these methods are not considered ``higher-order'' (i.e., $d>2$). However, the framework that we propose, similar to most of the literature, follows the structure of~\cite{nesterov2006cubic} (and~\cite{levenberg1944method, marquardt1963algorithm}) in terms of minimizing, in each iteration, a Taylor expansion of a certain order plus an appropriate regularization term. Recently, there has been a body of work following this structure with Taylor expansions of order higher than two~\cite{nesterov2021implementable, baes2009estimate, bubeck2019near, jiang2020unified, jiang2021optimal, grapiglia2022tensor}. Unlike our paper, these works are in the setting of convex optimization, do not study the complexity of minimizing the regularized Taylor expansion in each iteration (except in the case of $d=3$ for a subset of these papers), and derive sublinear rates of global convergence. There has also been work on lower bounds on the rates of convergence for such methods~\cite{Birgin2017, hazan_lowerbounds, carmon2020lower,nesterov2021implementable}. These lower bounds are nearly achieved by the algorithms in the aforementioned papers. \jeff{The recent textbook~\cite{cartis2022evaluation} provides an accessible summary of this literature and its broader scope. See also~\cite{cartis2019universal, cartis2020sharp, cartis2020concise} and references therein.}


In terms of work per iteration of higher-order Newton methods, Nesterov presents a polynomial-time algorithm in~\cite{nesterov2021implementable} for minimizing a quartically-regularized third-order Taylor expansion. This problem is revisited recently in~\cite{cartis2023second}, where an algorithm for recovering an approximate second-order point for a possibly nonconvex quartically-regularized third-order Taylor expansion is presented. In~\cite{silina2022unregularized}, a different third-order Newton method is presented which has polynomial work per iteration. In each iteration, this algorithm moves to a local minimum of the third-order Taylor expansion. It turns out that local minima of cubic polynomials can be found by semidefinite programs of polynomial size~\cite{ahmadi2022complexity}. To the best of our knowledge, no efficient algorithm for higher-order Newton methods of degree $d>3$ has been presented. In fact, designing such an algorithm is referred to as an open problem in~\cite[Sec. 1.5]{doikovthesis} and \cite[Sec. 5]{grapiglia2022tensor}. Interestingly, Nesterov asks in~\cite[Sec. 6]{nesterov2021implementable} whether it is possible to tackle this problem using ``some tools from algebraic geometry and the related technique of sums of squares''. This is precisely the approach that we take in this paper.

To our knowledge, the only works that establish superlinear rates of local convergence for higher-order Newton methods are~\cite{silina2022unregularized} and~\cite{Doikov2022} (and the related PhD thesis \cite{doikovthesis}), the latter of which came to our attention at the time of writing this paper. In~\cite{silina2022unregularized}, the authors establish third-order local convergence rate for an unregularized third-order Newton method applied to a strongly convex function. In~\cite{Doikov2022}, the authors establish superlinear local convergence for higher-order Newton methods applied to convex optimization problems with composite objective. When the smooth part of the objective function is strongly convex, the authors show local convergence of order $d$ in function value and norm of the subgradient for their proposed $d$\textsuperscript{th}-order Newton method. An algorithm carrying out the work per iteration of this method, however, is available only in the case of $d=3$ (and is the same as that in~\cite{nesterov2021implementable}). Moreover, similar to much of the literature, the regularization term that is added to the Taylor expansion in this method requires knowledge of the Lipschitz constant of the $d$\textsuperscript{th} derivative of $f$. Our proof technique for local superlinear convergence is different than~\cite{Doikov2022} both in the parts where the sum of squares programming aspects come in and in the parts that they do not. Furthermore, our method has polynomial work per iteration for \emph{any} degree $d$. It also does not rely on knowledge of any Lipschitz constants. Our regularization term is instead derived from the optimal value of a semidefinite program which can be written down from the coefficients of the Taylor expansion alone. This optimized approach can potentially lead to smaller deviations from the Taylor expansion and therefore an improved convergence factor. Finally, we note that in our work, assumptions on convexity of $f$ and knowledge of the Lipschitz constant of its $d$\textsuperscript{th} derivative are made only in Section~\ref{Sec: Global convergence}, where global convergence is established. Our approach in \Cref{Sec: Global convergence} is based on incorporating sum of squares methods into the framework of Nesterov in~\cite{nesterov2021implementable}, though in theory this can also be done with other globally convergent higher-order Newton methods. \jeff{In fact, at the time of revision, there has already been interesting follow-up work to our paper which combines our sum of squares framework with adaptive regularization techniques for tensor methods and analyzes the complexity of the resulting algorithm for finding an approximate stationary point of a nonconvex function~\cite{zhu2024global}.}

\subsection{Organization and Contributions}\label{SSec: Contributions and Organization}

In Section~\ref{Sec: Preliminaries}, we review preliminaries on sos-convexity, sos-convex polynomial optimization, and error rates of derivatives of Taylor expansions. In Section~\ref{Sec: Algorithm}, we present our main algorithm (\Cref{alg:main}). In \Cref{Sec: Main Proof}, we prove that our algorithm is well-defined in the sense that the semidefinite programs it executes are always feasible and that the next iterate is always uniquely defined (\Cref{thm:well defined}). We then prove that our semidefinite programming-based $d$\textsuperscript{th}-order Newton scheme has local convergence of order $d$ (\Cref{thm:rate}).
Compared to the classical Newton method, this leads to fewer calls to the Taylor expansion oracle (a common oracle in this literature; see e.g.,~\cite{bubeck2019near},~\cite[Sect. 2.2]{jiang2021optimal},~\cite[Sect. 1.1]{hazan_lowerbounds},~\cite[Sect. 2]{Birgin2017}, \jeff{\cite[Chap. 1.2]{cartis2022evaluation}}) at the price of requiring higher-order derivatives. The proof of Theorem~\ref{thm:rate} is more involved than the proof of local quadratic convergence of Newton's method. This is in part because the expression for the next iterate of Newton's method is explicit, whereas our next iterate comes from the solution to two semidefinite programs. We also remark that our proof framework is applicable to a broader class of higher-order Newton methods that may not necessarily use sum of squares techniques.

In \Cref{Sec: Numerical examples}, we present three numerical examples. We give an explicit expression and a geometric interpretation of our third-order Newton method in dimension one. We compare the basins of attraction of local minima for our higher-order methods to those of the classical Newton method. In~\Cref{Sec: Global convergence}, we present a slightly modified higher-order Newton method which is globally convergent under additional convexity and Lipschitzness assumptions similar to those in \cite{nesterov2021implementable}. This modified algorithm works in the case of $d$ being an odd integer and still has polynomial work per iteration and local convergence of order $d$. Finally, in~\Cref{sec:future}, we present a few directions for future research.

\section{Preliminaries}\label{Sec: Preliminaries}

\subsection{SOS-Convex Polynomial Optimization}\label{SSec: SOS-Convex Polynomial Optimization}
In each iteration of the higher-order Newton methods that we propose, two semidefinite programs (SDPs) need to be solved. These SDPs arise from the notion of sos-convexity, which is reviewed in this subsection.
\begin{definition}
A polynomial $p : \Rn \mapsto \R$ is said to be a \emph{sum of squares} (sos) if there exist polynomials $q_1,\dots,q_r : \Rn \mapsto \R$ such that $p = \sum_{i=1}^r q_i^2$.
\end{definition}

As is well known, one can check if a polynomial is sos by solving an SDP. The next theorem establishes this link. We denote that a symmetric matrix $A$ is positive semidefinite (i.e., has nonnegative eigenvalues) with the standard notation $A\succeq 0$.


\begin{theorem} [see, e.g., \citep{pablothesis}]\label{thm: sos sdp}
For a variable $x\in \Rn$ and an even integer $d$, let $\phi_{\frac{d}{2}}(x)$ denote the vector of all monomials of degree at most $\frac{d}{2}$ in $x$. A polynomial $p : \Rn \mapsto \R$ of degree $d$ is sos if and only if there exists a symmetric matrix $Q$ such that (i) $p(x) = \phi_{\frac{d}{2}}(x)^TQ\phi_{\frac{d}{2}}(x)$ for all $x \in \Rn$, and, (ii) $Q \succeq 0$.
\end{theorem}

The first constraint above can be written as a finite number of linear equations by coefficient matching. Therefore, the two constraints together represent the intersection of an affine subspace with the cone of positive semidefinite matrices. \jeff{Thus, as polynomials can be encoded as an ordered vector of coefficients, the set of sos polynomials of a given degree has a description as the feasible region of a semidefinite program.} Furthermore, the size of this SDP grows polynomially in $n$ when $d$ is fixed.



Throughout this paper, we denote the gradient vector (resp. Hessian matrix) of a function $g: \Rn \mapsto \R$ with the standard notation $\nabla g$ (resp. $\nabla^2 g$).

\begin{definition}[SOS-Convex]\label{Defn: sos convex}
A polynomial $p : \Rn \mapsto \R$ is said to be sos-convex if the polynomial $q: \Rn \times \Rn \mapsto \R$ defined as $q(x,y) \defeq y^T \nabla^2 p(x) y$ is sos.
\end{definition}

Note that any sos-convex polynomial is convex. The converse statement is not true, except for certain dimensions and degrees (see \citep{sosconvexity}). By \Cref{thm: sos sdp} above, the set of sos-convex polynomials of a given degree also form the feasible region of a semidefinite program. Because the polynomial $q(x,y)$ is quadratic in $y$, one can reduce the size of the underlying SDP. More specifically, a polynomial $p : \Rn \mapsto \R$ of degree $d$ is sos-convex\footnote{Note that an odd-degree polynomial can never be convex, except for the trivial case of affine polynomials.} if and only if 
there exists a symmetric matrix $Q\succeq 0$ such that $y^T \nabla^2 p(x) y=  (\phi_{\frac{d}{2}-1}(x) \otimes y)^T Q (\phi_{\frac{d}{2}-1}(x) \otimes y)$. (Here, $\otimes$ denotes the Kronecker product.) We see that the size of the SDP that represents sos-convex polynomials of degree $d$ in $n$ variables grows polynomially in $n$ when $d$ is fixed.



We next explain why sos-convex polynomial optimization problems can be solved with the first level of the so-called Lasserre hierarchy.
A polynomial optimization problem is a problem of the form
\begin{equation}\label{eq: pop}
\begin{aligned}
\inf_{x \in \Rn} \quad & g_0(x) \\
\st \quad & g_j(x) \leq 0 \quad j=1,\ldots,m,
\end{aligned}
\end{equation}
where $g_j(x)$ are real-valued polynomial functions of a variable $x \in \Rn$.
The first-level Lasserre relaxation (see \cite{Lasserre2000}) corresponding to problem \eqref{eq: pop} takes the form
\begin{equation}\label{eq: sos relaxation}
\begin{aligned}
\sup_{\gamma \in \R, \lambda \in \R^{m}} \quad & \gamma \\
\st \quad & g_0(x) - \gamma + \sum_{j=1}^{m} \lambda_j g_j(x) \text{ is sos} \\
& \lambda_j \geq 0\quad j=1,\ldots,m.
\end{aligned}
\end{equation}

\jeff{The reader can check that the optimal value of \eqref{eq: sos relaxation} is always a lower bound on that of \eqref{eq: pop}. The next theorem establishes that this lower-bound is tight when the defining polynomials of \eqref{eq: pop} are sos-convex.}


\begin{theorem}[See Corollary 2.5 from \citep{Lasserre2008}, and Theorem 3.3 from \citep{lasserre2009}]\label{thm: lasserre exact}
Suppose that the polynomials $g_0,\dots,g_m$ in \eqref{eq: pop} are sos-convex, the optimal value of \eqref{eq: pop} is finite, and that the Slater condition holds\footnote{That is, there exists some $\xbar \in \Rn$ such that $g_j(\xbar) < 0$ for all $j=1,\ldots,m$.}.
Then, the optimal values of \eqref{eq: pop} and \eqref{eq: sos relaxation} are the same.
Moreover, an optimal solution to \eqref{eq: pop} can be readily recovered from a solution to the semidefinite program that is dual to \eqref{eq: sos relaxation}.
\end{theorem}

This result is already proven by Lasserre in \citep{lasserre2009} using a lemma of Helton and Nie from~\citep{helton2010semidefinite}.
For completeness and for the benefit of the reader, we give an alternative short proof of the first claim.

\begin{proof}
Recalling that an sos polynomial is nonnegative and that $\lambda_j\geq0$ for $j=1,\ldots,m$, it is easy to see that the optimal value of \eqref{eq: pop} is larger than or equal to the optimal value of~\eqref{eq: sos relaxation}.
To show the opposite inequality, let $\gamma^*$ be the optimal value of \eqref{eq: pop}. Then, the convex function $x \mapsto g_0(x) - \gamma^*$ is nonnegative over the set $\{ x \mid g_j(x) \leq 0, j=1,\ldots,m\}$. By the convex Farkas lemma (see, e.g.,~\citep[Theorem 2.1]{s_lemma}), there exists a nonnegative vector $\lambda^* \in \R^{m}$ such that $p(x) \defeq g_0(x)- \gamma^* + \sum_{j=1}^{m} \lambda^*_j g_j(x) \geq 0$ for all $x \in \Rn$.
Notice that $p(x)$ is sos-convex since it is a conic combination of sos-convex polynomials.
Thus, by~\citep[Theorem 3.1]{sosconvexity}, the polynomial $q(x,y) \defeq p(y) - p(x) - \nabla p(x)^T (y-x)$ is sos.
Let $x^*$ be an optimal solution to \eqref{eq: pop} (such a vector must exist \cite{Belousov2002}). Observe that the polynomial $y \mapsto q(x^*,y)$ is also sos (since it is the restriction of $q(x,y)$ to $x=x^*$). By optimality of $x^*$ to \eqref{eq: pop}, we have $p(x^*) \leq 0$.
Since $p$ is nonnegative, we have $p(x^*) = 0$ and $\nabla p(x^*) = 0$.
Thus, $p(y) = q(x^*,y)$, and hence $p(y)$ must be sos.
Therefore, $\gamma^*,\lambda^*$ is feasible to \eqref{eq: sos relaxation}, and hence the optimal value of \eqref{eq: sos relaxation} is at least $\gamma^*$; i.e., the optimal value of \eqref{eq: pop}.

\end{proof}

For a proof of the second claim and an explicit expression of the dual of \eqref{eq: sos relaxation}, see Theorem 3.3 from \citep{lasserre2009}.


\subsection{Error rates of Taylor remainders}\label{SSec: Taylor remainders}
In this subsection, we review certain error rates of multivariate Taylor expansions that will be used in our arguments.
We denote by $\nabla^d f$ the $d$\textsuperscript{th} order symmetric tensor of order-$d$ partial derivatives of the function $f$.
We denote the tensor product of a set of vectors $x_1,\ldots,x_d\in\Rn$ with $x_1\boxtimes x_2 \boxtimes \ldots  \boxtimes x_d$.\footnote{We use this slightly nonstandard notation to avoid confusion with the Kronecker product.}
We use the notation $x^{\boxtimes d} $ to denote the tensor product of a vector $x\in\Rn$ with itself $d$ times.
With this notation, we can define 
the $d$\textsuperscript{th}-order Taylor expansion of a $d$-times differentiable function $f$ at a point $\bar{x}$ as
$$T_{\bar{x},d}(x) \mathrel{\mathop{:}}= f(\bar{x}) + \sum_{i=1}^d \frac{1}{i!}\langle \nabla^i  f(\bar{x}), (x-\bar{x})^{\boxtimes i} \rangle,$$
where $\langle \cdot , \cdot\rangle$ denotes the standard tensor inner product.
The \emph{remainder} or \emph{error} term of the Taylor expansion is
$$R_{\bar{x},d}(x) \mathrel{\mathop{:}}= f(x) - T_{\bar{x},d}(x).$$
For a $d$\textsuperscript{th}-order tensor $D$, let us define the following norm
$$\|D\| \defeq \underset{\|x_1\|, \ldots, \|x_d\|\le 1}{\max} \langle D , x_1\boxtimes x_2 \boxtimes \ldots  \boxtimes x_d \rangle,$$
where $||x_i||$ denotes the Euclidean 2-norm of the vector $x_i \in \Rn$. Note that for cases of $d = 1$ and $d=2$, this expression reduces to the standard Euclidean norm and the spectral norm, respectively.

We will need the following lemma in \Cref{Sec: Main Proof}.

\begin{lemma}[see, e.g., inequality (11) in~\cite{baes2009estimate}]
\label{Lem: Gradient remainder}
\label{Lem: Hessian Remainder}
Fix a vector $\bar{x}\in \Rn$.
Suppose $\nabla^d f$ has a Lipschitz constant $L$ over a convex set $C$ containing $\xbar$, i.e., 
\[
\| \nabla^d f(x) - \nabla^d f(y) \| \leq L \|x-y\|
\]
for all $x,y \in C$.
Then, for any $x \in C$, we have
\[
\|\grad R_{\xbar,d}(x)\| \leq \frac{L}{d!}\|x-\bar{x}\|^{d}.
\]
and 
\[
\|\Hess R_{\xbar,d}(x)\| \leq \frac{L}{(d-1)!}\|x-\bar{x}\|^{d-1}.
\]
\end{lemma}

\section{Algorithm Definition}\label{Sec: Algorithm}
For a given integer $d \geq 3$, we consider the task of minimizing a function $f$ which is assumed to have derivatives up to order $d$, and a local minimum $x^*$ satisfying $\Hess f(x^*) \succ 0$.
We also assume that the $d$\textsuperscript{th} derivative of $f$ is locally Lipschitz around the point $x^*$, i.e., there is a radius $r_L > 0$, and a scalar $L \geq 0$, such that for points $x,y$ in the set $\{z \in \Rn \mid \| z- x^*\| \leq r_L \}$, we have
\[
\| \nabla^d f(x) - \nabla^d f(y) \| \leq L \|x-y\|.
\]
Note that the latter assumption is always satisfied if the $d+1$\textsuperscript{th} derivative of $f$ exists and is continuous.
Our goal is to minimize $f$ by iteratively minimizing a surrogate function of the type
\[
T_{x_k,d}(x) + t||x-x_k||^{d'},
\]
where $x_k$ is our current iterate, $T_{x_k,d}$ is the Taylor expansion of $f$ of order $d$ at $x_k$, $d'$ is the smallest even integer greater than $d$ (as we require the surrogate to be a polynomial), and $t$ 
%
%
is chosen according to the following sum of squares program:

\begin{equation}\label{eq:min t}
\begin{aligned}
\min_{t \in \R} \quad & t \\
\st \quad & T_{x_k,d}(x) + t||x-x_k||^{d'} \quad \text{sos-convex} \\
& t \geq 0.
\end{aligned}
\end{equation}
In view of \Cref{thm: sos sdp} and the remarks after \Cref{Defn: sos convex}, this program can be reformulated as an SDP of size polynomial in $n$. Letting $t(x_k)$ denote the optimal value of \eqref{eq:min t} for a given $x_k$, we define our surrogate function to be
\begin{equation}\label{eq:psi}
\psi_{x_k,d}(x) \defn T_{x_k,d}(x) + t(x_k)||x-x_k||^{d'}.
\end{equation}
In our algorithm, we choose $x_{k+1}$ to be the minimizer of $\psi_{x_k,d}$ (which exists and is unique; see \Cref{thm:well defined} below).
By \Cref{thm: lasserre exact}, since $\psi_{x_k,d}$ is sos-convex, we can find its minimizer via another SDP of size polynomial in $n$.

If $x_k$ is far from $x^*$ so that $\nabla^2 f(x_k)$ is not positive definite, it may occur that \eqref{eq:min t} is infeasible.
If this occurs, we fix a positive scalar\footnote{Our analysis applies to any positive value of $\varepsilon$.} $\varepsilon$ and instead solve the SDP:
\begin{equation}\label{eq:min t bar}
\begin{aligned}
\min_{\bar{t} \in \R} \quad & \bar{t} \\
\st \quad & T_{x_k,d}(x) + \frac{1}{2}\bigg(\varepsilon-\lambda_{\min}\nabla^2 f(x_k)\bigg) ||x-x_k||^2 + \bar{t}||x-x_k||^{d'} \quad \text{sos-convex} \\
& t \geq 0.
\end{aligned}
\end{equation}
Let $\bar{t}(x_k)$ denote the optimal value of \eqref{eq:min t bar} and define
\begin{equation}\label{eq:psi bar}
\bar{\psi}_{x_k,d}(x) \defn T_{x_k,d}(x) + \frac{1}{2}\bigg(\varepsilon-\lambda_{\min}\nabla f(x_k)\bigg) ||x-x_k||^2 + \bar{t}(x_k)||x-x_k||^{d'}.
\end{equation}
We then let $x_{k+1}$ be the minimizer $\bar{\psi}_{x_k,d}$ (which again exists and is unique; see \Cref{thm:well defined} below).
As before, we can find a minimizer of $\bar{\psi}_{x_k,d}$ by solving an SDP of size polynomial in $n$; see \Cref{thm: lasserre exact}.

Our overall algorithm is summarized below: 

\begin{algorithm2e}[H]
\caption{$d$\textsuperscript{th}-order Newton method}
\label{alg:main}
\LinesNumbered
\SetKwInput{Parameter}{Parameter}
\SetKwInput{Input}{Input}
\DontPrintSemicolon
\Parameter{$\varepsilon > 0$}
\Input{$x_0 \in \Rn$}
\For{$k=0,\dots$}{
\eIf{$\nabla^2 f(x_k) \succ 0$}{
Solve \eqref{eq:min t} to find $t(x_k)$ \label{line:min t}\;
Let $x_{k+1}$ be the minimizer of $\psi_{x_k,d}$ (see \eqref{eq:psi}) \label{line: min psi}\;
}{
Solve \eqref{eq:min t bar} to find $\bar{t}(x_k)$ \label{line:min t bar}\;
Let $x_{k+1}$ be the minimizer of $\bar{\psi}_{x_k,d}$ (see \eqref{eq:psi bar})\label{line: min psi bar}\;
}
}
\end{algorithm2e}



\section{Algorithm Analysis and Convergence}\label{Sec: Main Proof}

In this section, we present our main technical results. Theorem~\ref{thm:well defined} shows that our algorithm is well-defined for all initial conditions.
Theorem~\ref{thm:rate} gives our convergence result.
We remind the reader that the assumptions made on $f$ are described in the first paragraph of \Cref{Sec: Algorithm}.
In particular, the function $f$ is not required to be convex, and the $d$\textsuperscript{th} derivatives of $f$ are not required to be globally Lipschitz.

\begin{theorem}\label{thm:well defined}
Algorithm~\ref{alg:main} is well-defined in the sense that
\begin{enumerate}[(i)]
\item the problems \eqref{eq:min t} and \eqref{eq:min t bar} are always feasible when required at Lines~\ref{line:min t} and \ref{line:min t bar}, and
\item the functions $\psi_{x_k,d}$ and $\bar{\psi}_{x_k,d}$ (see \eqref{eq:psi} and \eqref{eq:psi bar}) always possess a unique minimizer when required at Lines~\ref{line: min psi} and \ref{line: min psi bar}.
\end{enumerate}
\end{theorem}
\begin{theorem}\label{thm:rate}
There exist constants $r,c>0$ such that if $||x_0-x^*||\leq r$, then the sequence $\{x_k\}$ generated by \Cref{alg:main} satisfies $$||x_{k+1} - x^*|| \leq c ||x_k - x^*||^d$$ for all $k$. 
\end{theorem}

The power $d$ in this theorem is referred to as the \emph{order} of convergence and the constant $c$ is referred to as the \emph{factor} of convergence. We note that the factor of convergence arising from our proof is explicit. 

To prove Theorems \ref{thm:well defined} and \ref{thm:rate}, we first establish some technical lemmas. Lemmas~\ref{lem:sosconvex interior} and~\ref{lem:t bounded} are used to prove the first claim of Theorem~\ref{thm:well defined}; Lemmas~\ref{lem:matrix quadrature} and~\ref{lem: unique min} are for the second claim; and Lemmas~\ref{lem:t bounded}, \ref{lem:matrix quadrature}, and~\ref{lem:psi at optimum} are employed in the proof of Theorem~\ref{thm:rate}.

In Lemma~\ref{lem:sosconvex interior}, we show that a particular polynomial is in the interior of the cone of sos-convex polynomials.
This is used in Lemma~\ref{lem:t bounded} to show that we can always make our surrogate functions defined in~\eqref{eq:min t} and~\eqref{eq:min t bar} sos-convex.

\begin{lemma}\label{lem:sosconvex interior}
Let $x \defeq (x_1,\ldots,x_n)$. The polynomial
\[
p(x) = x^Tx + (x^Tx)^d
\]
is in the interior of the cone of sos-convex polynomials in $n$ variables and of degree at most $2d$.
\end{lemma}
\begin{proof}
We first establish the following claim:

\noindent {\bf Claim 0.} For all $d \geq 0$, the polynomial
 
\[
\tilde{p}_d(x) = 1 + (d+1)(x^Tx)^{d}
\]
can be written as $\phi_d(x)^TQ\phi_d(x)$, where $\phi_d$ is the standard basis of monomials of degree up to $d$ with the monomials appearing in ascending order of degree, and $Q$ is a positive definite matrix.


To prove Claim 0, it suffices to show that for all $d \geq 0$, there exists a constant $\alpha_d > 0$ and a positive definite matrix $\hat{Q}_d$ such that $1 + \alpha_d (x^Tx)^d=\phi_d(x)^T \hat{Q}_d\phi_d(x)$. Indeed, if $\alpha_d < d+1$, we can observe that
$$\tilde{p}_d(x) = 1 + \alpha_d(x^Tx)^d + \left((d+1)-\alpha_d\right)(x^Tx)^d = \phi_d(x)^T \left( \hat{Q}_d + Q'\right) \phi_d(x),$$ where $Q'$ can be taken to be positive semidefinite since $((d+1)-\alpha_d)(x^Tx)^d$ is sos.
If $\alpha_d > d+1$, we can observe that $$\tilde{p}_d(x) = \frac{d+1}{\alpha_d} + (d+1)(x^Tx)^d + (1-\frac{d+1}{\jeff{\alpha_d}})= \phi_d(x)^T \left( \frac{d+1}{\alpha_d} \hat{Q}_d + Q'\right) \phi_d(x),$$ where $Q'$ can be taken to be positive semidefinite since $(1-\frac{d+1}{\jeff{\alpha_d}})$ is sos.

Let us now proceed by induction on $d$ to prove the claim made in the previous paragraph. The case of $d = 0$ is clear since we can take any $\alpha_0 > 0$ and the associated matrix $\hat{Q}_0$ is simply a $1 \times 1$ matrix containing the scalar $1+\alpha_0$. Now suppose that the induction hypothesis holds for $d = k$. To construct $\alpha_{k+1}$ and $\hat{Q}_{k+1}$, 
we will add matrices associated with the polynomials $1 + \alpha_k (x^Tx)^k$ and $\alpha(x^Tx)^{k+1} - \alpha_k (x^Tx)^k$, where $\alpha$ is an arbitrary scalar.
From the induction hypothesis, there exist a scalar $\alpha_k > 0$ and a matrix $\hat{Q}_k \succ 0$ of size $\binom{n+k}{k} \times \binom{n+k}{k}$ that satisfy $$1 + \alpha_k (x^Tx)^k = \phi_k(x)^T \hat{Q}_k \phi_k(x) = \phi_{k+1}(x)^T \bmat \hat{Q}_k & 0 \\ 0 & 0\emat \phi_{k+1}(x).$$ Meanwhile, observe that we can write
$$\alpha(x^Tx)^{k+1} - \alpha_k(x^Tx)^k = \phi_{k+1}(x)^T\bmat 0 & A \\ A^T & \alpha P\emat\phi_{k+1}(x)$$ for some matrices $A$ and $P\succ 0$, where the zero block is of size $\binom{n+k}{k} \times \binom{n+k}{k}$. Indeed, we can take the matrix $P$ to be diagonal with its diagonal entries equalling the coefficients of $(x^Tx)^{k+1}$ and move the coefficients of $\alpha_k(x^Tx)^k$ to the matrix $A$.
Adding the two identities, we observe that:
\[
1 + \alpha(x^Tx)^{k+1} = \phi_{k+1}(x)^T\bmat \hat{Q}_k & A \\ A^T & \alpha P\emat\phi_{k+1}(x).
\]
Since $\hat{Q}_k$ and $P$ are both positive definite matrices, by the Schur complement condition, whenever $\alpha P - A^T\hat{Q}_k^{-1}A \succ 0$, the matrix on the right-hand side of the above expression will be positive definite.
One can therefore choose $\alpha_{k+1}$ to be any large enough value of $\alpha$ that satisfies the previous condition and let $\hat{Q}_{k+1} \defn \bmat \hat{Q}_k & A \\ A^T & \alpha_{k+1} P\emat$.
We have thus proved Claim 0.



By Claim 0 (with $d$ replaced by $d-1$), we can fix a positive definite matrix $Q$ such that $1+d(x^Tx)^{d-1} = \phi(x)_{d-1}^T Q \phi_{d-1}(x)$ for all $x$.
One can check that 
\begin{align*}
y^T \nabla^2 p(x) y &= y^T \bigg( 2I + 2d(x^Tx)^{d-1}I + 4d(d-1)(x^Tx)^{d-2} xx^T\bigg)y \\
&= 2 (y^Ty) (1+d(x^Tx)^{d-1}) + 4d(d-1)(x^Tx)^{d-2}(x^Ty)^2\\
&= 2 (y^Ty) \phi(x)_{d-1}^T Q \phi_{d-1}(x) + 4d(d-1)(x^Tx)^{d-2}(x^Ty)^2\\
&= (\phi_{d-1}(x) \otimes y)^T (Q \otimes 2I + Q') (\phi_{d-1}(x) \otimes y),
\end{align*}
where $Q'$ can be taken to be positive semidefinite since $4d(d-1)(x^Tx)^{d-2}(x^Ty)^2$ is sos.
Since the matrix $Q \otimes 2I + Q'$ is positive definite, it follows that $p$ is in the interior of the cone of sos-convex polynomials of degree at most $2d$.
\end{proof}

\begin{lemma}\label{lem:t bounded}
Suppose $f:\Rn \mapsto \R$ has continuous derivatives up to order $d$ over a compact set $B \subseteq \Rn$.
If $\nabla^2 f(x) \succ 0$ for all $x \in B$, then $t(x)$ (i.e., the optimal value of~\eqref{eq:min t}) is \jeff{uniformly} bounded \jeff{from above} over $B$. 
\end{lemma}
\begin{proof}

Let $\delta$ be a positive scalar such that $\lambda_{\min} \nabla^2 f(x) \geq \delta$ for all $x \in B$.
Let $x'$ be any vector in $B$, and define
\[
F_{x'}(x) \defn \frac{2}{\delta} T_{x',d}(x'+x).
\]
Since $\nabla^2 f(x') \succeq \delta I$, we have $\nabla^2 F_{x'}(0) \succeq 2I$.
Let $Q_{x'}$ (resp. $C_{x'}$) be the sum of the quadratic and higher (resp. cubic and higher) terms of $F_{x'}$.
For a polynomial $p$, define $||p||_\infty$ as the infinity norm of the coefficients of $p$ when expressed in the standard monomial basis.
By \Cref{lem:sosconvex interior}, we can fix a positive scalar $R$ such that for any polynomial $p$ of degree at most $d$ with $||p||_\infty \leq R$, we have that the polynomial $||x||^2 + ||x||^{d'} + p$ is sos-convex.
Fix a scalar $M$ such that $||C_{x'}||_\infty < M$ for all $x' \in B$.
Define $\alpha \defn \min \{ 1,\frac{R}{M} \}$.
We have $||x \mapsto C_{x'}(\alpha x)||_\infty \leq \alpha^3 ||C_{x'}||_\infty$ since all terms of $C_{x'}$ are of cubic or higher order.
Then we can write
\begin{align*}
\frac{1}{\alpha^2}Q_{x'}(\alpha x) + \|x\|^{d'} &= \frac{1}{2} x^T \nabla^2 F_{x'}(0) x + \frac{1}{\alpha^2}C_{x'}(\alpha x) + \|x\|^{d'} \\
&= \frac{1}{2} x^T (\nabla^2 F_{x'}(0) - 2I) x + \frac{1}{\alpha^2}C_{x'}(\alpha x) + (\|x\|^2 + \|x\|^{d'}).
\end{align*}
Since $\nabla^2 F_{x'}(0) \succeq 2I$, the first term is sos-convex.
We can bound the second term as follows: $\|x \mapsto \frac{1}{\alpha^2}C_{x'}(\alpha x)\|_\infty \leq \alpha ||C_{x'}||_\infty \leq \alpha M \leq R$.
Thus, the sum of the second and the third term is sos-convex by the definition of $R$.
It follows that the polynomial 
\[
\frac{1}{\alpha^2}Q_{x'}(\alpha x) + \|x\|^{d'} \text{ is sos-convex.}
\]
We can then conclude the sos-convexity of the polynomials
\begin{enumerate}[(a)]
    \item $Q_{x'}(\alpha x) + \alpha^2 \|x\|^{d'}$,
    \item $F_{x'}(\alpha x) + \alpha^2 \|x\|^{d'}$,
    \item $F_{x'}(\alpha x) + \alpha^{2-d'} \| \alpha x\|^{d'}$,
    \item $F_{x'}(x) + \alpha^{2-d'} \| x\|^{d'}$,
    \item $T_{x',d}(x'+x) + \frac{\delta}{2} \alpha^{2-d'} \| x\|^{d'}$, and
    \item $T_{x',d}(x) + \frac{\delta}{2} \alpha^{2-d'} \| x- x'\|^{d'}$,
\end{enumerate}
respectively (a) by scaling, (b) by the observation that the affine terms do not affect sos-convexity, (c) by rewriting, (d) by a linear change of coordinates, (e) by another rescaling, and (f) by an affine change of coordinates.
Thus, we have $t(x)  \leq \frac{\delta}{2} \alpha^{2-d'}$ for $x\in B$.

\end{proof}

We next use a quadrature rule for integration to establish a technical lemma that is needed for the remainder of this section. By a polynomial matrix, we mean a matrix whose entries are polynomial functions.



\begin{lemma}\label{lem:matrix quadrature}
Let $M: \R \mapsto \Snn$ be univariate polynomial matrix whose entries have degree at most $d$, where $d$ is even. Suppose $M(s) \succeq 0$ for all $s \in [0,1]$. Then,

\[
\int_0^1 M(s) ds \succeq \frac{1}{2(d^2-1)} M(\alpha)
\]
for $\alpha \in \{0,1\}$.
\end{lemma}
\begin{proof}
Using a quadrature rule for integration proposed in~\citep{Clenshaw1960} and analyzed in~\citep{Imhof1963}, 
there exist a set of weights $w_0,\dots,w_d\geq 0$, with $w_{0} = \frac{1}{d^2-1}$, and a set of points $s_0,\dots,s_d \in [-1,1]$, with $s_{0} = 1$, such that for any polynomial $p$ of degree at most $d$ we have
\[
\int_{-1}^{1} p(s) ds = \sum_{i=0}^{d} w_i p(s_i).
\]
Now we can write
\begin{align*}
\int_0^1 M(s) ds &= \frac{1}{2} \int_{-1}^{1} M\left(\frac{1-s}{2}\right)ds \\
&= \frac{1}{2} \sum_{i=0}^{d} w_i M\left(\frac{1-s_i}{2}\right) \\
& \succeq \frac{1}{2} w_0 M\left(\frac{1-s_0}{2}\right)  =  \frac{1}{2(d^2-1)} M(0).
\end{align*}
By replacing $s$ with $1-s$, the claim with $\alpha=1$ follows.
\end{proof}

The next lemma directly proves the second claim of Theorem~\ref{thm:well defined} and is possibly of independent interest.

\begin{lemma}\label{lem: unique min}
If a convex polynomial $p:\Rn \mapsto \R$ satisfies $\nabla^2 p(x_0) \succ 0$ for any point $x_0 \in \Rn$, then $p$ has a unique minimizer.
\end{lemma}
\begin{proof}
Without loss of generality, assume $x_0 = 0$.
Let $d$ be an even integer greater than the degree of the Hessian of $p$.
For any $x \in \Rn$, 
\begin{align*}
p(x) &= p(0) + x^T\nabla p(0) + x^T \left( \int_0^1 \int_0^t \nabla^2 p(sx) ds dt \right) x \\
&= p(0) + x^T\nabla p(0) + x^T \left( \int_0^1 t\int_0^1 \nabla^2 p(stx) ds dt \right) x \\
&\geq p(0) + x^T\nabla p(0) + x^T \left( \int_0^1 t\left( \frac{1}{2(d^2-1)} \nabla^2 p(0) \right) dt \right) x \\
&= p(0) + x^T\nabla p(0) + \frac{1}{4} x^T \left( \frac{1}{d^2-1} \nabla^2 p(0) \right)  x,
\end{align*}
where the inequality follows from \Cref{lem:matrix quadrature}.
Thus, $p$ is lower bounded by a coercive\footnote{We recall that a function $g:\Rn \mapsto \R$ is coercive if $g(x)\rightarrow\infty$ as $||x||\rightarrow\infty$.} quadratic function, and hence $p$ is coercive itself. A coercive function that is convex (and hence continuous) has at least one minimizer.

Suppose for the sake of contradiction that $p$ had two minimizers $\bar{x}, \bar{y}$. Then, by convexity, any point on the line segment connecting $\bar{x}$ and $\bar{y}$ would also be a minimizer. Since $p$ is a polynomial, it follows that $p$ must be constant along the line passing through $\bar{x}$ and $\bar{y}$. This contradicts coercivity.  
\end{proof}


We remark that the statement of \Cref{lem: unique min} does not hold for non-polynomial convex functions (consider, e.g., the univariate function $\max\{0,x^2 -1\}$).

The next lemma is used in the proof of Theorem~\ref{thm:rate}.


\begin{lemma}\label{lem:psi at optimum}
There exists a constant $r>0$ such that if $||x_k-x^*||\leq r,$ then $\lambda_{\min}\nabla^2 \psi_{x_k,d} (x^*) \geq \frac{1}{2}\lambda_{\min}\nabla^2f(x^*)$.
\end{lemma}
\begin{proof}
We show that we can take $$r = \min\left \{ r_L,\left(\frac{(d-1)! \lambda_{\min}\nabla^2f(x^*)}{2L}\right)^{\frac{1}{d-1}}\right \}$$ where $r_L$ and $L$ are as in the first paragraph of \Cref{Sec: Algorithm}.
By \Cref{Lem: Hessian Remainder}, For every $x$ satisfying $\|x-x_k\|\ \le r_L$, we have 
$$\|\nabla^2 f(x) - \nabla^2 T_{x_k,d}(x)\| \le \frac{L}{(d-1)!}\|x-x_k\|^{d-1}.$$
Thus, if $\|x^*-x_k\| \le r$, we have $$\|\nabla^2 f(x^*) - \nabla^2 T_{x_k,d}(x^*)\| \le \frac{1}{2}\lambda_{\min}\nabla^2f(x^*).$$

It follows that
\[
\lambda_{\min} \nabla^2 T_{x_k,d}(x^*) \geq \frac{1}{2}\lambda_{\min}\nabla^2f(x^*).
\]
Indeed, if there was a unit vector $y$ such that if $y^T\nabla^2 T_{x_k,d}(x^*) y < \frac{1}{2}\lambda_{\min}\nabla^2f(x^*)$, the previous inequality would be violated.

Recall from \eqref{eq:psi} that $\psi_{x_k,d}$ is obtained by adding to $T_{x_k,d}$ the convex function $t(x_k)\|x-x_k\|^{d'}$.
Therefore, we have $\nabla^2 \psi_{x_k,d}(x^*) \succeq \nabla^2 T_{x_k,d}(x^*),$ which gives the claim.




\end{proof}



We now have all the ingredients to prove Theorems~\ref{thm:well defined} and~\ref{thm:rate}.

\begin{proof}[Proof of \Cref{thm:well defined}]
\ \\
(i) When $\nabla^2 f(x_k) \succ 0$, the proof of \Cref{lem:t bounded} with $B = \{ x_k\}$ demonstrates a feasible solution to \eqref{eq:min t}.
This argument also extends to show feasibility of \eqref{eq:min t bar} since the polynomial $T_{x_k,d}(x) + \frac{1}{2}\big(\varepsilon-\lambda_{\min}\nabla^2 f(x_k)\big) ||x-x_k||^2$ has a positive definite Hessian at $x_k$.
\ \\
(ii) At \Cref{line: min psi} (resp. \Cref{line: min psi bar}), $\psi_{x_k,d}$ (resp. $\bar{\psi}_{x_k,d}$) has a positive definite Hessian at $x_k$.
Moreover, the polynomial $\psi_{x_k,d}$ (resp. $\bar{\psi}_{x_k,d}$) is sos-convex and therefore convex.
Thus, by \Cref{lem: unique min}, $\psi_{x_k,d}$ (resp. $\bar{\psi}_{x_k,d}$) has a unique minimizer.

\end{proof}

\begin{proof}[Proof of \Cref{thm:rate}]
Since $d > 1$, it suffices to show that there exist constants $r',c'>0$ such that if $||x_0-x^*||\leq r'$, then $||x_1 - x^*|| \leq c' ||x_0 - x^*||^d$.

By continuity of the map $x \mapsto \lambda_{\min}\nabla^2 f(x)$, there exists a scalar $r_1 > 0$ such that $\lambda_{\min} \nabla^2 f(x) \geq \frac{1}{2} \lambda_{\min} \nabla^2 f(x^*) > 0$ for all $x$ with $||x-x^*|| \leq r_1$.

Let $r_2 > 0$ be the constant needed for the conclusion of \Cref{lem:psi at optimum} to hold.
Define \[r' \defn \min \{r_L,r_1,r_2 \}\] and $\Omega \defn \{x \in \Rn \mid ||x-x^*|| \leq r'\}$.
Suppose $x_0 \in \Omega$.
Note that in this case, \Cref{alg:main} finds the next iterate $x_1$ by minimizing the polynomial $\psi_{x_0,d}$ defined in $\eqref{eq:psi}$.

By the fundamental theorem of calculus, we have
\[
\nabla \psi_{x_0,d}(x^*) - \nabla \psi_{x_0,d}(x_1) = \left( \int_0^1 \nabla^2 \psi_{x_0,d}(x_1 + s(x^*-x_1)) ds \right) (x^*-x_1) .
\]
Since $x_1$ minimizes $\psi_{x_0,d}$, we have $\nabla \psi_{x_0,d}(x_1) = 0$, and thus
\[
\nabla \psi_{x_0,d}(x^*) = \left( \int_0^1 \nabla^2 \psi_{x_0,d}(x_1 + s(x^*-x_1)) ds \right) (x^*-x_1) .
\]
We can bound the norm of this vector from below:
\begin{equation}\label{eq:grad ftoc}
||\nabla \psi_{x_0,d}(x^*)|| \geq \lambda_{\min} \left( \int_0^1 \nabla^2 \psi_{x_0,d}(x_1 + s(x^*-x_1)) ds \right) ||x^*-x_1||.
\end{equation}
Applying first \Cref{lem:matrix quadrature} and then \Cref{lem:psi at optimum}, we have
\begin{align*}
\lambda_{\min} \left( \int_0^1 \nabla^2 \psi_{x_0,d}(x_1 + s(x^*-x_1)) ds \right) &\geq \frac{\lambda_{\min}\nabla^2 \psi_{x_0,d}(x^*)}{2((d'-2)^2-1)} \\
&\geq \frac{\lambda_{\min}\nabla^2f(x^*)}{4((d'-2)^2-1)}.
\end{align*}
Substituting this into \eqref{eq:grad ftoc} and rearranging yields
\begin{equation}\label{eq:grad psi}
\|x_1 - x^*\| \leq \frac{4((d'-2)^2-1)}{\lambda_{\min}\nabla^2f(x^*)} \|\nabla \psi_{x_0,d} (x^*)\|.
\end{equation}
Expanding $\nabla \psi_{x_0,d}(x^*)$, we have
\begin{align*}
\|\nabla \psi_{x_0,d}(x^*)\| &= \left\|\nabla T_{x_0,d}(x^*) + \nabla ( t(x_0) ||x-x_0||^{d'}) \bigg|_{x^*} \right\| \\
& = \left\|\nabla T_{x_0,d}(x^*) + t(x_0) d' ||x^*-x_0||^{d'-2} (x^*-x_0) \right\| \\
&\leq ||\nabla T_{x_0,d}(x^*)|| + t(x_0) d' ||x^*-x_0||^{d'-1}.
\end{align*}
Applying \Cref{Lem: Gradient remainder} and noting that $\nabla f(x^*) = 0$, we have
\[
||\nabla \psi_{x_0,d}(x^*)|| \leq \frac{L}{d!} ||x^*-x_0||^d + t(x_0) d' ||x^*-x_0||^{d'-1}.
\]
Using \Cref{lem:t bounded} and the fact that $||x^*-x_0|| \leq r'$, we get
\begin{align*}
||\nabla \psi_{x_0,d}(x^*)|| &\leq \frac{L}{d!} ||x^*-x_0||^d + (\sup_{x \in \Omega} t(x)) d' \max\{r',1\} ||x^*-x_0||^d \\
&= \left(\frac{L}{d!}+ (\sup_{x \in \Omega} t(x)) d'\max\{r',1\} \right) ||x^*-x_0||^d.
\end{align*}
Substituting into \eqref{eq:grad psi}, we have
\[
||x_1 - x^*|| \leq \left( \frac{4((d'-2)^2-1)}{\lambda_{\min}\nabla^2f(x^*)} \left(\frac{L}{d!}+(\sup_{x \in \Omega} t(x)) d'\max\{r',1\}\right) \right) ||x^*-x_0||^d
\]
as desired.
\end{proof}

\section{Numerical Examples}\label{Sec: Numerical examples}
We present three examples to compare the performance of our $d$\textsuperscript{th}-order Newton methods and the classical Newton method.
\subsection{The Univariate Case}\label{SSec: Univariate examples}

In the univariate case, the iterations of the classical Newton method read $$x_{k+1}=x_k-\frac{f^\prime(x_k)}{f^{\prime\prime}(x_k)}.$$ 
In terms of finding a root of $f'$, this iteration can be interpreted as first computing the first-order Taylor expansion of $f'$ at $x_k$, and then finding the root of this affine function to define $x_{k+1}$.

We derive a similar explicit formula for our higher-order Newton method in the case where $n=1$, $d=3$, and $f''$ is positive. Since convex univariate polynomials are sos-convex, finding explicit solutions to the two SDPs involved in each iteration of our algorithm reduces to arguments about roots of univariate polynomials.

\begin{proposition}\label{prop:third formula}
In the univariate case, when $f''(x_k) > 0$ and $f'''(x_k) \neq 0$, the next iterate of the $3$\textsuperscript{rd}-order version of Algorithm~\ref{alg:main} is given by\footnote{Note that when $f''(x_k) > 0$ and $f'''(x_k) = 0$, the third-order Taylor series is convex and coincides with the second-order Taylor series. Therefore, the next iterates of the third-order and the classical Newton method coincide.}
\[
x_{k+1} = x_k - 2\frac{f''(x_k)}{f'''(x_k)} - \sqrt[3]{\frac{f'(x_k) - \frac{2}{3} \frac{(f''(x_k))^2}{f'''(x_k)}}{\frac{(f'''(x_k))^2}{12f''(x_k)}}}.
\]
\end{proposition}

\begin{proof}
To simplify notation, we let $T \defeq T_{x_k,3}$ and $\psi \defeq \psi_{x_k,3}$.
By translation, we may assume $x_k = 0$.
Then $T(x) = f(x_k) + xf'(x_k) + \frac{1}{2}x^2f''(x_k) + \frac{1}{6} x^3 f'''(x_k)$,
$\psi (x) = T(x) + tx^4$, where $t$ is the smallest constant that makes $\psi$ convex.
We have $\psi''(x) = f''(x_k) + xf'''(x_k) + 12tx^2$.
The discriminant of $\psi''$ is $(f'''(x_k))^2 - 48tf''(x_k)$, which tells us that $t = \frac{(f'''(x_k))^2}{48f''(x_k)}$.

To find $x_{k+1}$, we look for the root of $\psi'$.
One can write the expression for $\psi'$ in the following form:
\[
\psi'(x) = \frac{(f'''(x_k))^2}{12f''(x_k)}\left(x+2\frac{f''(x_k)}{f'''(x_k)}\right)^3 + f'(x_k) - \frac{2}{3}\frac{(f''(x_k))^2}{f'''(x_k)}.
\]
Observe that a univariate cubic polynomial of the form $a(x-b)^3 + c$, with $a \neq 0$, has a unique root at $x = b - \sqrt[3]{\frac{c}{a}}$.
Therefore, after a translation back by $x_k$, we have \[x_{k+1} = x_k-2\frac{f''(x_k)}{f'''(x_k)} - \sqrt[3]{\frac{f'(x_k) - \frac{2}{3} \frac{(f''(x_k))^2}{f'''(x_k)}}{\frac{(f'''(x_k))^2}{12f''(x_k)}}}.\]
\end{proof}

As in the case of the classical Newton method, the expression in \Cref{prop:third formula} can be interpreted geometrically in terms of finding a root of $f'$.
This iteration computes the second-order Taylor expansion of $f'$ at $x_k$, adds a sufficiently large cubic term to enforce monotonicity, and then finds the root of this monotone cubic function to define $x_{k+1}$.


\paragraph{Example 1}
In this example, we apply our method to the univariate function
\beq \label{Eq; square root function} f(x)=\sqrt{x^2+1}-1.\eeq
This is a strictly convex function with its unique minimizer at $x^*=0$.
One can check that the classical Newton method converges to this minimizer if and only if $|x_0| < 1$.
Using \Cref{prop:third formula}, we can calculate the exact basin of convergence of our third-order Newton method to be $(-\beta, \beta)$, where
\[\beta = \sqrt{\frac{1}{3}\left( 11 + \frac{142}{\sqrt[3]{1691 + 9i\sqrt{47}}} + \sqrt[3]{1691 + 9i\sqrt{47}} \right)} \sim 3.407.\]
This is strictly larger than the basin of convergence of the classical method.

\Cref{fig: Newton Comparison} demonstrates the difference between one iteration of the classical and our third-order Newton method starting at the point $x_0=1.5$.
We display the quadratic and quartic polynomials $T_{x_0,2}$ and $\psi_{x_0,3}$.
The minimizers of these polynomials are denoted by $x_1^{\textrm{Newton}}$ and $x_1^{\textrm{3ON}}$, which are respectively the next iterates of the classical and our third-order Newton method.
Since the third-order Taylor expansion of $f$ provides a more accurate approximation, we see that the next iterate of our method is closer to $x^*$, while that of the classical Newton method moves farther away from $x^*$.

\begin{figure}[H]
    \centering
    \includegraphics[width=.6\textwidth]{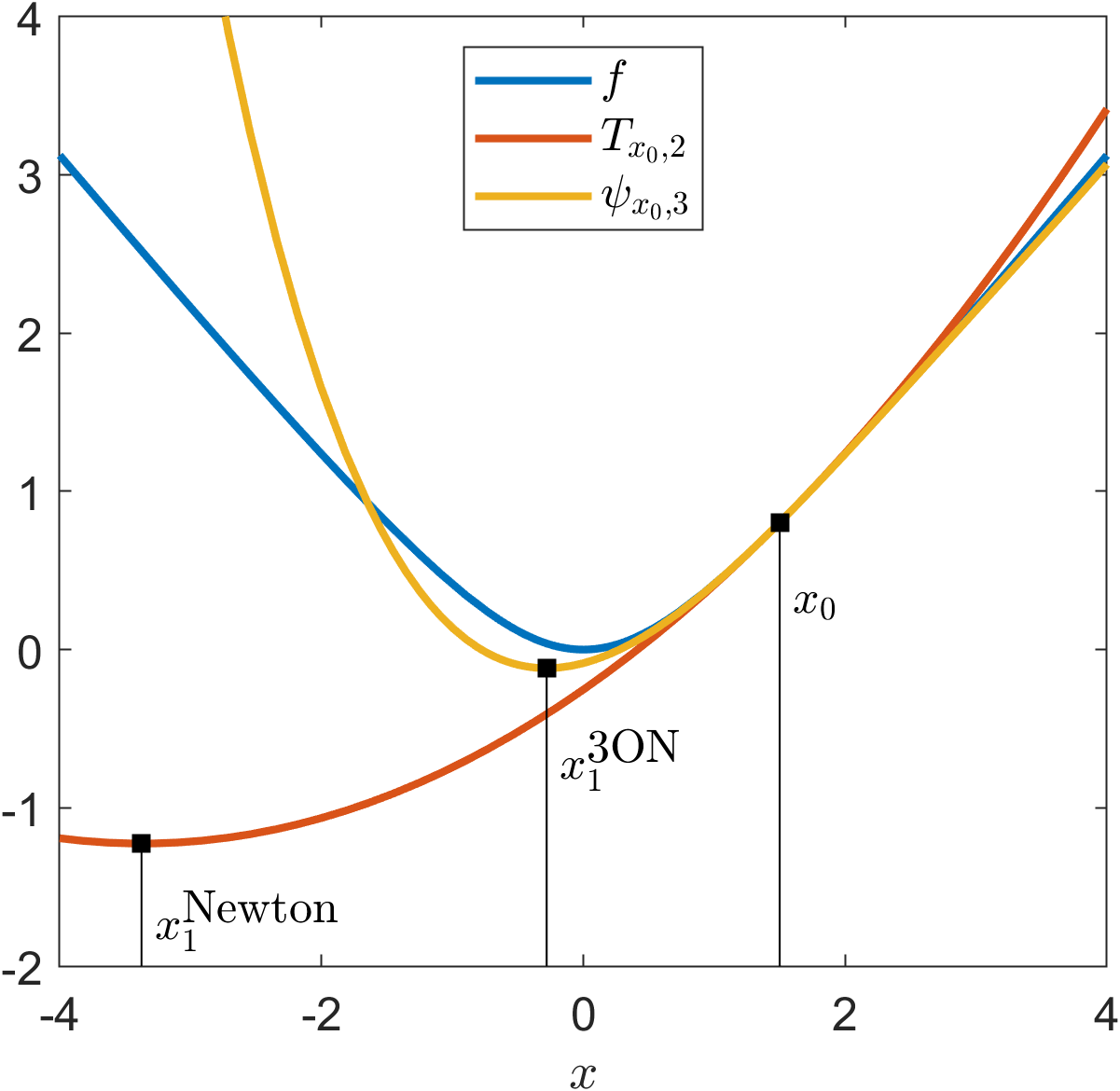}
    \caption{A comparison of one iteration of the classical Newton method and our third-order Newton method applied to the function in \eqref{Eq; square root function} starting at $x_0= 1.5$.}
    \label{fig: Newton Comparison}
\end{figure}

For our $d$\textsuperscript{th}-order Newton methods with $d>3$, we calculate the radii of convergence numerically.
These radii increase with degree as the following table demonstrates:
\begin{center}
\begin{tabular}{||c c||} 
 \hline
 Degree $d$ & Radius of Convergence \\ [0.5ex] 
 \hline\hline
 2 (Classical Newton) & 1 \\ 
 \hline
 3 & $\sim$3.4 \\
 \hline
 4 & $\sim$4.5 \\
 \hline
 5 & $\sim$5.9 \\
 \hline
\end{tabular}
\end{center}

We can visualize the speed of convergence of the fifth-order method, for example, in \Cref{fig: Newton Speed}.
In this figure, we plot the absolute value of $|x_k-x^*|$ starting at $x_0 = 5.9$, which is close to the boundary of the basin.
In just five iterations, the method reaches a point with absolute value approximately $10^{-15}$.

\begin{figure}[H]
    \centering
    \includegraphics[width=.57\textwidth]{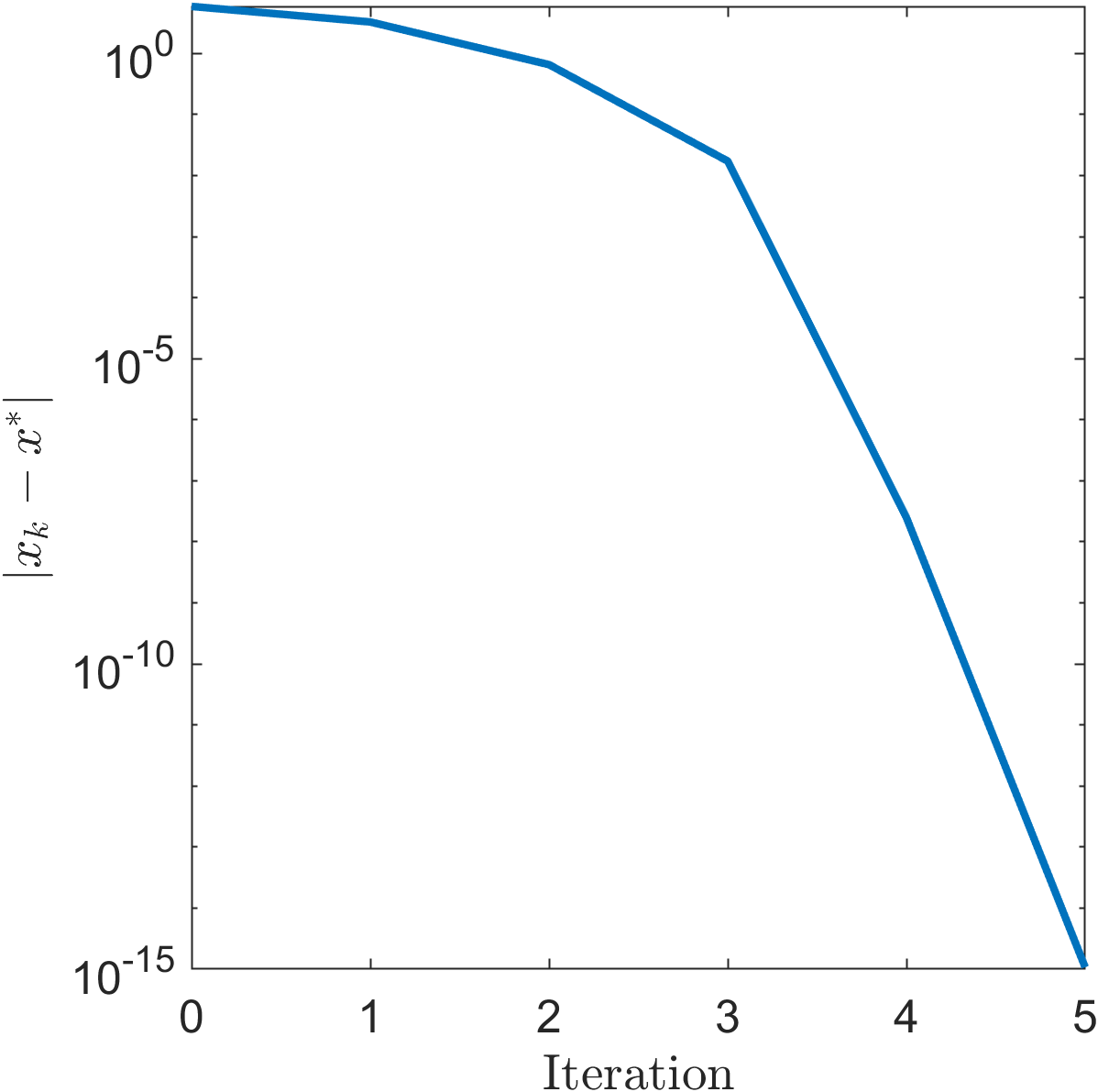}
    \caption{5\textsuperscript{th}-order Newton iterates applied to the function in \eqref{Eq; square root function}.}
    \label{fig: Newton Speed}
\end{figure}
\paragraph{Example 2}\label{Ex: 363 function}
In this example, we compare our third-order method to the classical Newton method when applied to the function
\beq \label{Eq: arctan function}
f(x) = 2x \arctan(x)-\log(1+x^{2})+\frac{1}{10}x^{2}.
\eeq
This is a strongly convex function with its unique minimizer at $x^*=0$.

In \Cref{fig: arctan}, $N_2$ (resp. $N_3$) is the map that takes a point to the corresponding next iterate of the classical (resp. third-order) Newton method.
In this example, the third-order method satisfies $|N_3(x)| < |x|$ for all nonzero $x$, implying global convergence of the method.
Meanwhile, the classical Newton method oscillates between $\pm 13.494$ when $x_0$ is outside of the range $[ -\alpha, \alpha]$, where $\alpha \sim 1.712$ is point of intersection of the functions $N_2$ and $-x$.

In \Cref{fig: Newton Arctan Comparison}, we can see a comparison of the iterates of the third-order and the classical Newton method starting from the initial condition $x_0=1.7$.
While both methods converge to the minimizer, the third-order method converges much faster.

\begin{figure}[H]
     \centering
     \begin{subfigure}[b]{0.49\textwidth}
         \centering
         \includegraphics[width=\textwidth]{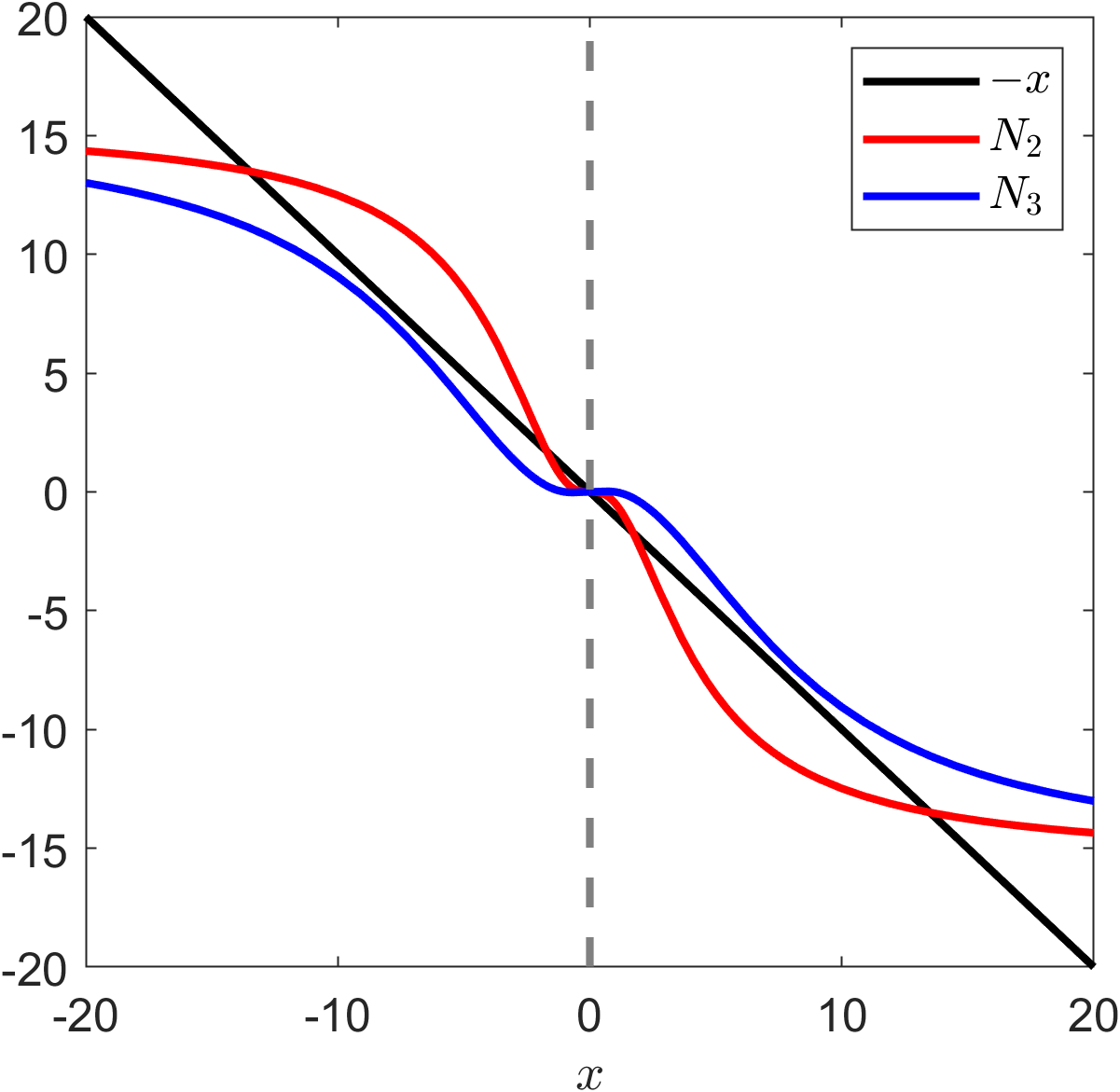}
         \caption{}
     \end{subfigure}
     \hfill
     \begin{subfigure}[b]{0.49\textwidth}
         \centering
         \includegraphics[width=\textwidth]{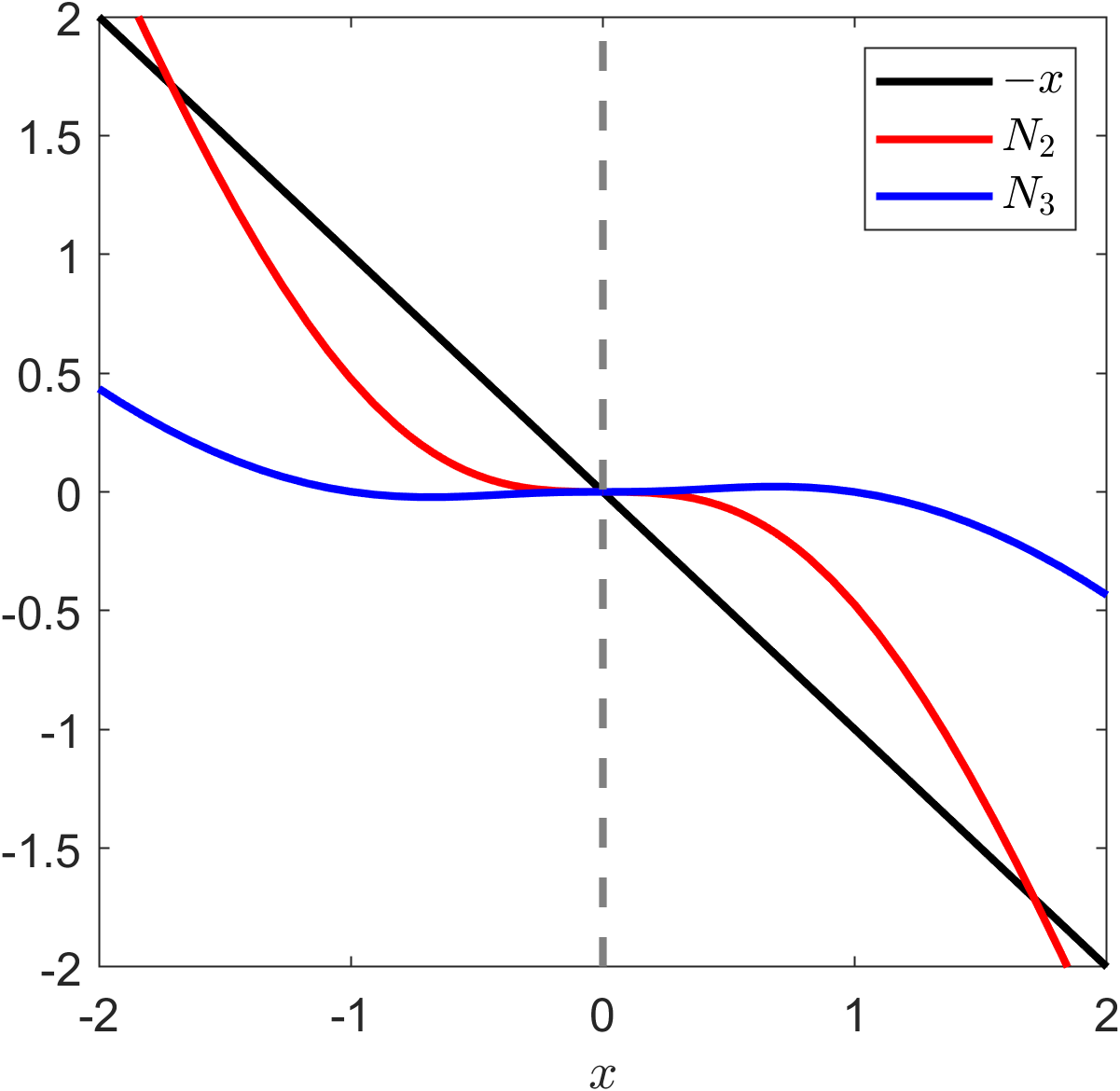}
         \caption{}
     \end{subfigure}
        \caption{Comparison of the classical Newton map $N_2$ and our third-order Newton map $N_3$ applied to the function in \eqref{Eq: arctan function}. Subfigure (a) implies that the third-order method is globally convergent, while the classical method is not. Subfigure (b) zooms in on the behavior of these maps near the origin to show that the basin of attraction for the classical method is approximately $( -1.712, 1.712)$.}
        \label{fig: arctan}
\end{figure}

\begin{figure}[H]
    \centering
    \includegraphics[width=.6\textwidth]{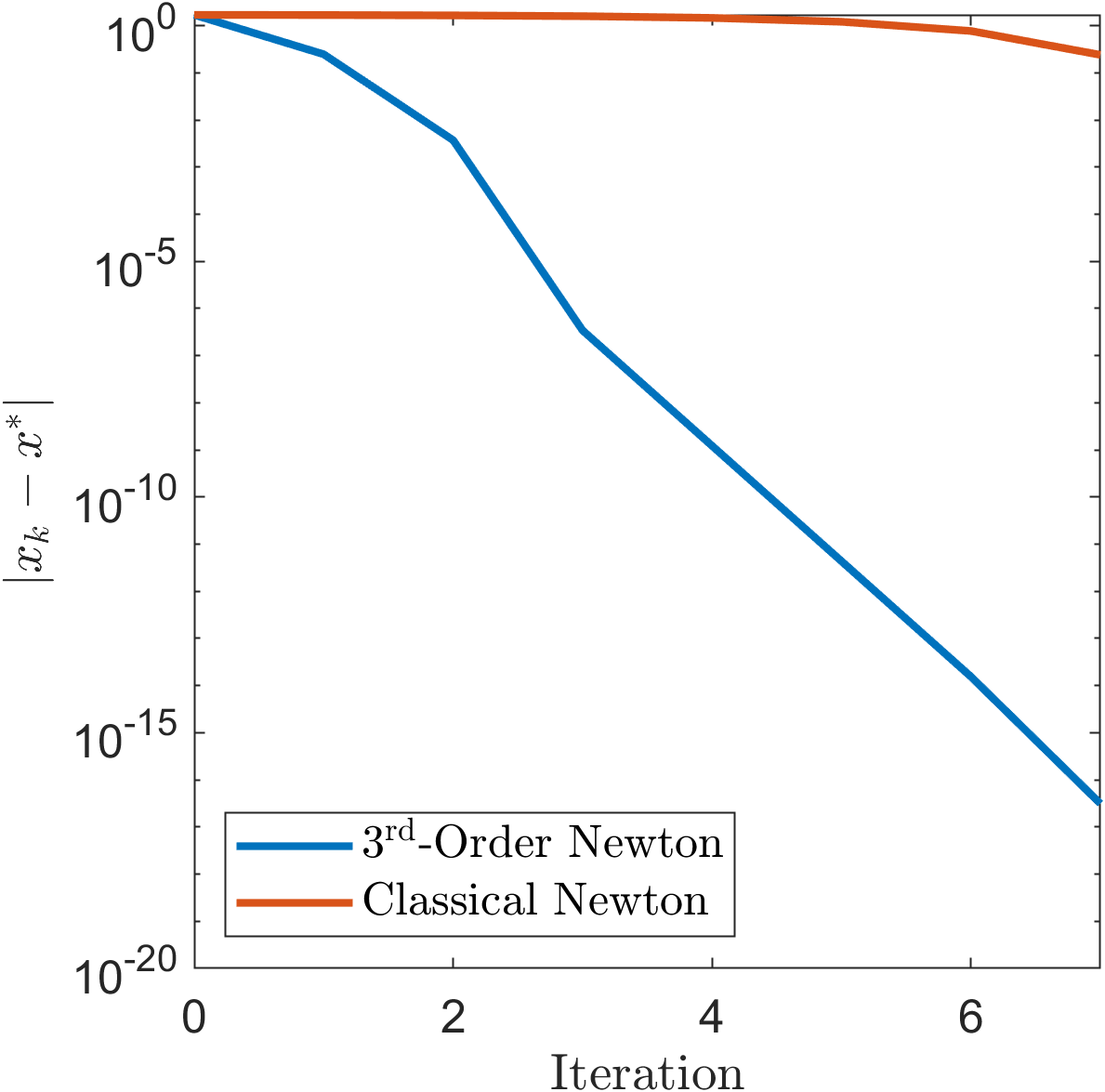}
    \caption{Iterates of our third-order and the classical Newton method applied to the function in \eqref{Eq: arctan function} starting from a point in the basin of attraction of both methods.}
    \label{fig: Newton Arctan Comparison}
\end{figure}

\subsection{A Multivariate Example}
In our last example, we compare the classical and the third-order Newton methods applied to a standard test function in nonlinear optimization called the Beale function:
\[
f(x_1,x_2) = (1.5 - x_1 + x_1x_2)^2 + (2.25 - x_1 + x_1x_2^2)^2 + (2.625 - x_1 + x_1x_2^3)^2.
\]
This nonconvex function has a single global minimum at $x^*=(3,0.5)^T$ and no other local minima.
In \Cref{fig: Beale}, we explore the behavior of both methods with initial conditions in the region $\{x \in \R^2 \mid  \|x\|_\infty \leq 4\}$.
We initialize the classical method and our third-order method at a fine grid of points in this box and run both methods for $350$ iterations.
For our third-order method, we take the parameter $\varepsilon$ in \Cref{alg:main} to be equal to $0.01$.
In \Cref{fig: Beale}, the color yellow corresponds to initial points that converge to $x^*$, and the color blue corresponds to any other behavior including divergence or convergence to a point which is not a local minimum.
In this example, the two basins are incomparable, but that of the third-order method is more contiguous and larger in volume.
\begin{figure}[H]
     \centering
     \begin{subfigure}[b]{0.49\textwidth}
         \centering
         \includegraphics[width=\textwidth]{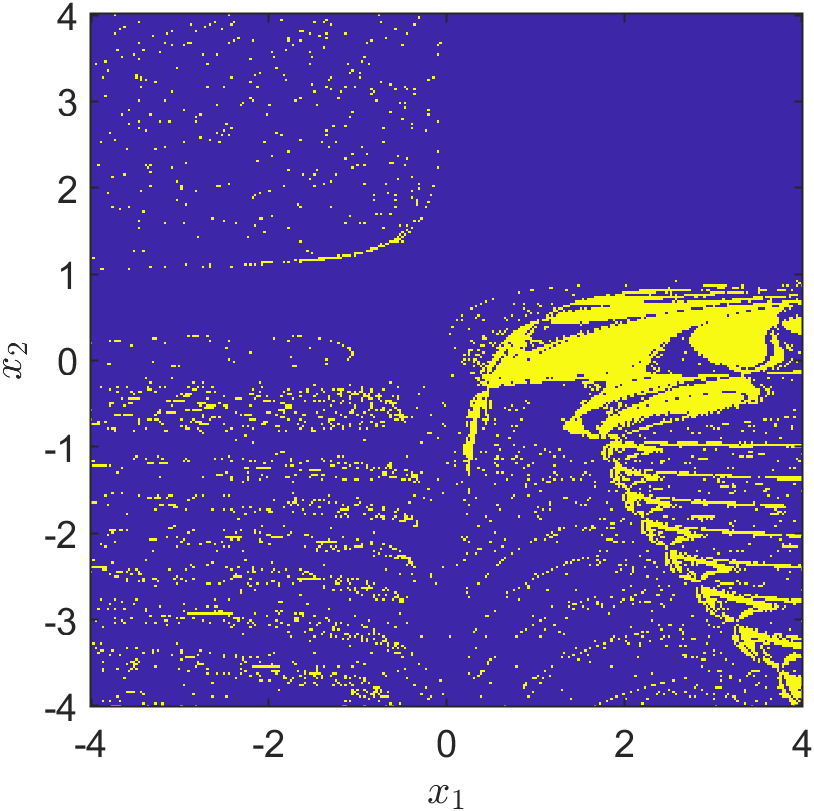}
         \caption{classical Newton}
     \end{subfigure}
     \hfill
     \begin{subfigure}[b]{0.49\textwidth}
         \centering
         \includegraphics[width=\textwidth]{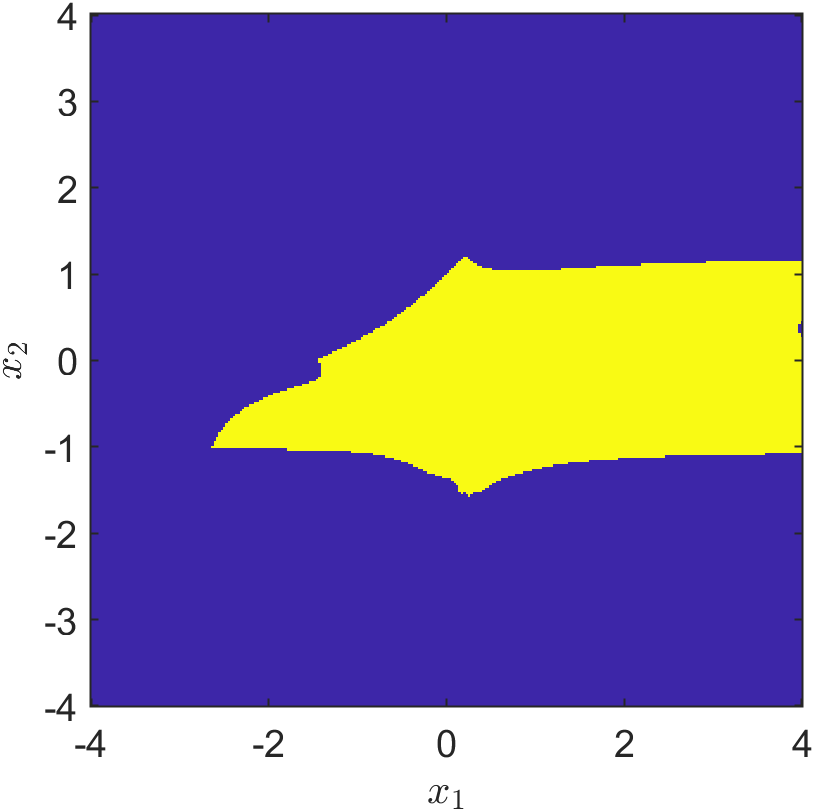}
         \caption{Third-order Newton}
     \end{subfigure}
        \caption{The basins of attraction for the classical and the third-order Newton methods for the minimizer of the Beale function. The basin for the classical method has fractal structure, demonstrating more sensitivity to initialization.}
        \label{fig: Beale}
\end{figure}

\section{Global convergence}\label{Sec: Global convergence}
In this section, we present a slightly modified algorithm which has global convergence under additional assumptions. There is a vast literature on modifications to Newton's method that lead to global convergence in special circumstances: see, e.g.,~\cite{nesterov2006cubic, ortega2000iterative, Mor1982RecentDI, conn2000trust}.
In the setting of our work, it turns out that we can use a result of Nesterov from \citep{nesterov2021implementable} to show that a simple modification to our algorithm that still has polynomial work per iteration is globally convergent when the Taylor expansion is made to an odd order.\footnote{The reason we need the Taylor expansion order to be odd is that in the work of Nesterov, the Taylor polynomial is regularized by a term of degree one larger. We need this new term to be a polynomial function for sum of squares methods to be readily applicable.} This modified algorithm (\Cref{alg: global} below) also inherits the local convergence order of \Cref{alg:main}.

As in \citep{nesterov2021implementable}, suppose the $d$\textsuperscript{th} derivative of the function $f:\Rn \mapsto \R$ that we wish to minimize has a Lipschitz constant $L_d$, and that an upper bound $M$ on $L_d$ is known. In this setting, consider the following algorithm:

\begin{algorithm2e}[htb]
\caption{$d$\textsuperscript{th}-order globally convergent Newton method ($d$ odd)}
\label{alg: global}
\LinesNumbered
\SetKwInOut{Input}{Input}
\DontPrintSemicolon
\Input{$x_0 \in \Rn$}
\For{$k=0,\dots$}{
Solve \eqref{eq:min t} to find $t(x_k)$\;
Let $x_{k+1}$ be the minimizer of $T_{x_k,d}(x) + \max \{ \frac{dM}{(d+1)!},t(x_k) \} \|x-x_k\|^{d+1}$\;
}
\end{algorithm2e}

Using the same arguments as those in the proof of \Cref{thm:well defined}, one can see that the next iterate $x_{k+1}$ produced by this algorithm is well-defined whenever $\nabla^2 f(x_k) \succ 0$. Also as before, problem \eqref{eq:min t} can be solved as a semidefinite program of size polynomial in the dimension. This claim also holds for the problem of finding the (unique) minimizer of the degree $d+1$ polynomial \[T_{x_k,d} + \max \{ \frac{dM}{(d+1)!},t(x_k) \} \|x-x_k\|^{d+1}.\] This is because the polynomials $\|x-x_k\|^{d+1}$ and $T_{x_k,d} +t(x_k)\|x-x_k\|^{d+1}$ are sos-convex and a conic combination of two sos-convex polynomials is sos-convex, making \Cref{thm: lasserre exact} applicable. 

\begin{theorem}\label{thm: global}
    Suppose $f:\Rn\rightarrow\R$ has bounded level sets, a positive definite Hessian everywhere, and the Lipschitz constant of its $d$\textsuperscript{th} derivative bounded above by $M$.\footnote{The assumptions that we make here are the same as those in \citep{nesterov2021implementable} except that our assumption of positive definiteness of the Hessian is stronger than the assumption of positive semidefiniteness of the Hessian made in \citep{nesterov2021implementable}.} Then, the iterates of Algorithm~\ref{alg: global} starting from any $x_0\in~\Rn$ converge to the (unique) minimizer of $f$. Furthermore, Algorithm~\ref{alg: global} has local convergence rate of order $d$.
\end{theorem}


\begin{proof}
Since the Hessian of $f$ is positive definite everywhere, the function $f$ is strictly convex.
This, along with boundedness of the level sets, implies that $f$ has a unique (global) minimizer which we call $x^*$.

Define $\psi_{x_k,d}(x) \defn T_{x_k,d}(x) + \max \{ \frac{dM}{(d+1)!},t(x_k) \} \|x-x_k\|^{d+1}$.
By Theorem~1 from \citep{nesterov2021implementable}, we have $\psi_{x_k,d}(x) \geq f(x)$ for all $x \in \Rn$, thus the method is monotone; i.e., $f(x_{k+1}) \leq f(x_k)$.
Let $M_k \defn \max \left \{ M,\frac{(d+1)!t(x_k)}{d} \right \}$ and $\delta_k \defn f(x_k)-f(x^*)$.
Since the set $\{ x \in \Rn \mid f(x) \leq f(x_0)\}$ is compact and the method is monotone, there exists a scalar $D$ such that $\|x_k-x^*\| \leq D$ for all $k$.
By the arguments in the proof of Theorem 2 from~\citep{nesterov2021implementable}, we can conclude that
\[
\delta_k - \delta_{k+1} \geq C_k \delta_k^{\frac{d+1}{d}},
\]
where $C_k \defn \frac{d}{d+1} \left( \frac{d!}{(dM_k + L_d)D^{d+1}} \right)^{\frac{1}{d}}$.

By \Cref{lem:t bounded}, we know that \[t_{\max} \defn \sup_{\|x-x^*\|\leq D} t(x)\] is finite.
Letting $M_{\max} \defn \max \{ M,\frac{(d+1)!t_{\max}}{d} \}$, we have $M_k \leq M_{\max}$, and therefore $C_k \geq \frac{d}{d+1} \left( \frac{d!}{(dM_{\max} + L_d)D^{d+1}} \right)^{\frac{1}{d}}$ for all $k$.
Continuing the argument from the proof of Theorem~2 from~\cite{nesterov2021implementable}, we can conclude that
\[
f(x_k) - f(x^*) \leq \frac{(dM_{\max} + L_d)D^{d+1}}{d!} \left( \frac{d+1}{k} \right)^d.
\]
Thus, we have $f(x_k) - f(x^*) \rightarrow 0$ and therefore $x_k \rightarrow x^*$.

For the local superlinear convergence rate, it suffices to show that for $x_k$ close enough to $x^*$, we have
\[
||x_{k+1} - x^*|| \leq c' ||x_k - x^*||^d
\] for some constant $c'$.
Let $r_1$ and $r_2$ be as in the proof of \Cref{thm:rate}, $r' \defn~\min \{ r_1,r_2\},$ and $\Omega \defn \{x \in \Rn \mid ||x-x^*|| \leq r'\}$.
By the arguments in the proof of \Cref{thm:rate}, for every $x_k\in \Omega$, we have
\begin{align*}
||\nabla \psi_{x_k,d}(x^*)|| &\leq \frac{L_d}{d!} ||x^*-x_k||^d + \max \left \{ \frac{dM}{(d+1)!},t(x_k) \right \} (d+1) ||x^*-x_k||^d \\
&\leq \frac{L_d}{d!} ||x^*-x_k||^d + \max \left \{ \frac{dM}{(d+1)!},\sup_{x \in \Omega}t(x) \right \} (d+1) ||x^*-x_k||^d.
\end{align*}
Substituting into \eqref{eq:grad psi} (with $x_0$ replaced with $x_k$), we have
\[
||x_{k+1} - x^*|| \leq  c'  ||x^*-x_k||^d,
\]
where 
\[
c' \defn \frac{4((d-1)^2-1)}{\lambda_{\min}\nabla^2f(x^*)} \left( \frac{L_d}{d!} + \max \left \{ \frac{dM}{(d+1)!},\sup_{x \in \Omega}t(x) \right \} (d+1) \right).
\]
We note that by \Cref{lem:t bounded}, $\sup_{x \in \Omega}t(x)$ is finite.
\end{proof}

\section{Future directions}\label{sec:future}
Besides the question of extending the results of \Cref{Sec: Global convergence} to the case of $d$ even, there are a few other potential directions for future research that we wish to highlight:
\begin{itemize}

\item Can we replace the SDPs used in \Cref{alg:main} with more scalable conic programs such as linear programs (LPs) or second-order cone programs (SOCPs)?
There has been work (see, e.g., \cite{dsos_sdsos}) on replacing methods based on sos programming with LP or SOCP-based approaches that rely on more tractable subsets of sos polynomials, such as the so-called \emph{diagonally dominant sum of squares} (dsos) or \emph{scaled diagonally dominant sum of squares} (sdsos) polynomials.
In our setting, we might wish to replace the constraint in \eqref{eq:min t} (or \eqref{eq:min t bar}) that a polynomial is sos-convex with a constraint that it is ``dsos-convex'' or ``sdsos-convex'' (see, e.g.,~\cite{dc_sos}).
The results in~\cite{dc_sos} on the difference of dsos-convex decompositions of arbitrary polynomials could be explored to potentially replace the first SDP in each iteration of \Cref{alg:main} with an LP or SOCP. One would then need to establish an appropriate dsos or sdsos version of \Cref{thm: lasserre exact} to replace our second SDP with an LP or SOCP. It would be interesting to compare the factor of convergence of such an algorithm to that of the SDP-based approach.



\item Can we create a method that uses a sparse subset of higher-order derivatives of the function $f$ and that perhaps approximates the remaining derivatives in order to speed up each iteration? Such a method would be a higher-order analogue to the so-called ``quasi-Newton'' methods which rely on approximations of the Hessian of $f$ (see, e.g.,~\cite[Chap. 6]{Nocedal_Wright_2006}). An example of such a higher-order quasi-Newton method which results in semidefinite programs of small size in each iteration has been proposed in~\cite{spq}, but its convergence properties are currently unknown.

\item Can we use our method or a modification thereof to solve systems of nonlinear equations (in a way that is superior to simply minimizing the sum of the squares of the equations)?
The classical Newton method and its variants can be used for this purpose (see, e.g.,~\cite[Sect. 11.1]{Nocedal_Wright_2006}). What are the right higher-order analogues of these approaches?

\item Each iteration of the algorithms that we have presented in this paper can be interpreted as running just one iteration of the so-called ``convex-concave procedure'' (see, e.g., \cite{boyd2016ccp}) to a particular difference of convex decomposition of the Taylor expansion of $f$. Are there benefits of working with alternative difference of convex decompositions (see, e.g., \cite{dc_sos}) of the Taylor expansion, or running more iterations of the convex-concave procedure before the Taylor polynomial is updated?
\end{itemize}

\section*{Acknowledgements}
We would like to thank Jean-Bernard Lasserre for insightful discussions around the results in~\cite{lasserre2009}.

\bibliography{newton}

\end{document}